\documentclass[12pt,reqno]{amsart}
\usepackage{amsmath}
\usepackage{amssymb}
\usepackage{amsxtra}
\usepackage{amsthm}
\usepackage{color}
\usepackage{slashed}

\usepackage{hyperref}
\hypersetup{
    colorlinks = true,
    colorlinks = false,
    linkcolor = blue,
    anchorcolor = red,
    citecolor = blue,
    filecolor = red,
    pagecolor = red,
    urlcolor = blue
}
\usepackage{graphicx}

\usepackage{latexsym}
\usepackage{graphicx}

\def\ee{\end{equation}}
\def\bea{\begin{eqnarray}}
\def\eea{\end{eqnarray}}

\def\bz{\bar z}
\def\bs{\bar s}
\def\p{\partial}
\def\bp{\bar\partial}

\def\bb{\bar b}

\def\d{\delta}

\newcommand{\szego}{Szeg\"o\ }

\renewcommand{\Re}{{\operatorname{Re}\,}}
\renewcommand{\Im}{{\operatorname{Im}\,}}

\renewcommand{\epsilon}{\varepsilon}

\newcommand{\var}{{\operatorname{Var}}}
\newcommand{\Var}{{\bf Var}}

\newcommand{\inv}{^{-1}}
\newcommand{\kahler}{K\"ahler }

\newcommand{\PP}{{\mathbb P}}

\newcommand{\R}{{\mathbb R}}
\newcommand{\C}{{\mathbb C}}

\newcommand{\kcalomega}{\mathcal{K}_{[\omega_0]}}

\newcommand{\CP}{\C\PP}
\renewcommand{\d}{\partial}
\newcommand{\dbar}{\bar\partial}
\newcommand{\ddbar}{\partial\dbar}

\newcommand{\E}{{\mathbf E}}

\newcommand{\half}{{\textstyle \frac 12}}
\newcommand{\vol}{{\operatorname{Vol}}}

\newcommand{\bcal}{\mathcal{B}}

\newcommand{\dcal}{\mathcal{D}}

\newcommand{\fcal}{\mathcal{F}}

\newcommand{\ical}{\mathcal{I}}

\newcommand{\ncal}{\mathcal{N}}

\newcommand{\pcal}{\mathcal{P}}

\newcommand{\ep}{\varepsilon}
\newcommand{\de}{\delta}

\newcommand{\om}{\omega}

\newtheorem{theo}{{\sc Theorem}}[section]

\newtheorem{maindefn}{{\sc Definition}}

\newtheorem{prop}[theo]{{\sc Proposition}}

\title{Heat kernel measures on  random surfaces}

\begin{document}

\author{Semyon Klevtsov$^{1,3}$ and Steve Zelditch$^{2,3}$}

\maketitle
{\small
\address{\it$^1$Mathematisches Institut, Universit\"at zu K\"oln, Weyertal 86-90, 50931 K\"oln, Germany}

\address{\it $^2$Department of Mathematics, Northwestern  University, Evanston, IL 60208, USA}

\address{\it $^3$Simons Center for Geometry and Physics, Stony Brook University, NY 11794, USA}

\vspace{.2cm}
}

\begin{center}

\email{\tt\footnotesize  sam.klevtsov@gmail.com, zelditch@math.northwestern.edu}
\end{center}

\begin{abstract}

The heat kernel  on the symmetric
space of positive definite Hermitian matrices
is used to endow the spaces of Bergman
metrics  of degree $k$ on a Riemann surface $M$ with a family of probability measures depending on a choice of
the background metric. Under a certain matrix-metric correspondence, 
each positive definite Hermitian matrix corresponds to a \kahler metric on $M$. The one and two point
functions of the random metric are calculated in a variety of limits as $k$ and $t$ tend to infinity. 
In the limit when the time $t$ goes to infinity the fluctuations of the random metric
around the background metric are the same as the fluctuations of random zeros of holomorphic sections. This is due to the fact that the random zeros form the boundary of the space of Bergman metrics.

\end{abstract}


\section{Introduction} 

In a recent series of articles \cite{FKZ1,FKZ2}, the authors have been investigating a new approach to defining  `random surfaces'.  The main idea is to define integrals over the infinite dimensional  space $\kcalomega$ of  metrics of fixed area $2 \pi$ 
in a fixed conformal class $[\omega_0]$
on a Riemann surface $M$ as limits 
\begin{equation} \label{INT1} \int_{\kcalomega} F(g) e^{- S(g)} \dcal
g := \lim_{k \to \infty} \int_{\bcal_k} F_k(g) e^{- S_k(g)}
\dcal_k g
\end{equation}
of integrals over  finite dimensional spaces $\bcal_k$ of  {\it Bergman metrics}.  Given a background metric\footnote{With some abuse of notation we make no distinction everywhere between the metric $g$ and the corresponding \kahler form $\omega$, connected as $\omega=ig_{a\bb}dz^a\wedge d\bz^{\bb}$.} $\omega_0$ and a choice of 
a basis $\{s_j(z)\}$ of holomorphic sections of $L^k\to M$, the spaces $\bcal_k$
can be identified with the
non-positively curved  symmetric
space $\pcal_{N_k} :=SL(N_k, \C)/SU(N_k)$ of
 positive definite Hermitian matrices.  The
 general   question is to find
 sequences $\{d\mu_k =  e^{- S_k(g)}
\dcal_k g\}$ of measures on $\bcal_k$ which are independent of the choice of
the basis $\{s_j(z)\}$, which vary in a simple way under the change of the reference point $\omega_0\in\kcalomega$ and have good asymptotic properties
as $k \to \infty$. It would be particularly interesting to construct a sequence $\{d\mu_k\}$ which tends to Liouville theory measure on metrics of fixed area $\int_M \omega_0$,
although that is not the aim of the present article.

The sequence of measures we study in this article are the  heat kernel  measures
 \begin{equation} \label{HEAT} d\mu_k^t(P): = p_k(t, I, P) dV(P), \end{equation} 
 where $dV(P)$ is Haar measure,  $p_k(t, P_1, P_2)$ is the heat kernel of the symmetric space $\pcal_{N_k}$ and $I$ is the identity matrix. Under the matrix-metric identification $\bcal_k \simeq \pcal_{N_k}$ the identity matrix corresponds to the background metric $\omega_{\phi_I}$ and the heat kernel measure \eqref{HEAT}  is transported to $\bcal_k$. The measure is invariant under the action of the unitary group $U(N_k)$. Hence it is invariant of the choice of the basis of sections used to identify metrics and matrices.   Then \eqref{HEAT} is the probability measure on $\bcal_k$ induced by 
 Brownian motion on $\pcal_{N_k}$ starting at the identity $I$ and continuing up to time $t$.
The heat kernel measure is almost canonical, the  only
choices being the time $t$ and the  background metric $\omega_{\phi_I}$ used to make
the identification and to start the Brownian motion. The purpose of this article is to study the behavior of the heat
kernel measure \eqref{HEAT} on $\bcal_k$ as $k \to \infty$. The main geometric 
quantities we study are the {\it area statistics} \begin{equation} \label{AREA} X_U(\omega) = \int_U \omega \end{equation}
measuring the area of an open set $U \subset M$ with respect to the random area
form $\omega \in \bcal_k$. We determine the means and variances of these random
variables and their smooth analogues  $X_{f} (\omega)= \int_M f \omega $ 
with $f \in C^{\infty}(M)$  in various regimes, e.\ g.\ when the time  $t = t_k$ is allowed to vary with $k$. The calculations  are valid for any choice of
background metric and the dependence on the background metric is simple
and explicit.

The heat kernel measure \eqref{HEAT}  is $U(N_k)$-invariant in the $P$
 variable.  Such invariant measures have generic form
$d\mu_{\bcal_k}(P):=\mathcal F_{\bcal_k}(e^{\lambda} )d\mu_{\rm Haar}(P),
$ where $ \mathcal F_{\bcal_k}(e^{\lambda})$ is a function of the eigenvalues of $P$. It
 was shown in  \cite{FKZ2} that the eigenvalue density
 $\fcal_{\bcal_k}(e^{\lambda})$ induces a function $\fcal_{k, 2}(\nu_1,
\nu_2)$ on $\R_+^2$, so that the 2-point correlation function has the form,
\begin{equation}
\label{final2p} K_{2, k} (z_1, z_2): = 
\E_k\, \phi_P(z_1)\phi_P(z_2)=\phi_I(z_1)\phi_I(z_2)+\frac1{k^2} I_{2,k}(\rho),
\end{equation}
 where \begin{equation}\label{rhointro} \rho(z_1,z_2) = \frac{|B_k(z_1, z_2)|^2}{B_k(z_1, z_1) B_k(z_2, z_2)}  \end{equation} is an important invariant of the \szego
kernel  $B_k(z_1, z_2)$  of the background metric, known as the Berezin kernel.
Thus, $I_{2,k}(z_1, z_2)$  is the bi-potential of the variance of the area forms (or \kahler metrics in higher dimensions) relative to the exterior tensor product $\omega_0 \boxtimes \omega_0$,
\begin{equation}\label{vc2}
\Var\big(\omega_{\phi}\big) =  \E\big(\omega_{\phi}
\boxtimes \omega_{\phi}
\big) -
\E\big(\omega_{\phi}\big)\boxtimes \E
\big(\omega_{\phi}\big) = \E\big(\omega_{\phi}
\boxtimes \omega_{\phi}
\big) -
\omega_{0}\boxtimes
\omega_{0},
\end{equation} 
in the sense that
\begin{equation}\label{varcur2}{\bf Var}\big(\omega_{\phi}\big)=\frac1{k^2}
 (i\ddbar)_z\,(i\ddbar)_w \,I_{2, k} (z,w),
\;.\end{equation} 
The general formula for $I_{2,k}(\rho)$ for any $U(N_k)$-invariant measure is,
 \begin{equation} 
 \label{I2}
 I_{2,k}( \rho)  = 
\frac1{2} \int_{\R_+^2} \int_0^\pi \log (\nu_1^2 \cos^2 \beta + \nu_2^2 \sin^2 \beta)
 \log \frac{A
+ \sqrt{A^2 - B^2}}{2}\, \fcal_{2, k}(\nu) \sin\beta\,d\beta d\nu_1 d \nu_2, \end{equation}
with
$$\left\{ \begin{array}{l}
A =
(\nu_1^2 \cos^2 \beta + \nu_2^2 \sin^2 \beta) \rho+(\nu_1^2 \sin^2 \beta + \nu_2^2 \cos^2 \beta) (1 - \rho),\\ \\
 B = 
2(\nu_1^2 - \nu_2^2)\sqrt{\rho (1 - \rho)} \cos \beta\sin \beta .
\end{array} \right.. $$
The transform
$\fcal_{\bcal_k}(e^{\lambda}) \to \fcal_{2, k}(\nu)$  is very  difficult to evaluate, and
we do not know how to do so directly
even for the heat kernel measure. 
The first term of \eqref{final2p} is the potential of $\omega_{\phi_I}(z_1) \omega_{\phi_I}(z_2)$
where $\omega_{\phi_I}$ is the background metric, and the second term $I_{2,k}(\rho)$
is the correction to this term, which we call the variance term. The key point
is that $I_{2, k}(\rho)$ is a function only of the variable $\rho$.
This  result defnes a  transform $$\fcal_{\bcal_k}(e^{\lambda}) 
\to I_{2, k}(\rho)$$ from  eigenvalue densities to  variance  terms depending only
on $\rho$. It would
be interesting to know if this transform is invertible in some sense, so that one can
construct $U(N_k)$-invariant measures with prescribed pair correlation functions.
In this article we calculate $I_{2,k}(\rho)$ when $\fcal_{\bcal_k}(e^{\lambda})$ comes
from the heat kernel measure, by a different method  (also
used in \cite{FKZ2}).

\subsection{\label{MAIN} Main results}

In the case of heat kernel measures we  calculate
the pair correlation function explicitly not just for Riemann surfaces, but
for general projective \kahler manifolds  (see \S \ref{twopoint}). 
The calculations give an explicit formula for the variance term of the 2-point function $I_{2, k} (t, \rho)$: 
\begin{eqnarray}\label{MAINintro}
\p_\rho I_{2, k}(t, \rho)
=\frac{2t}\rho-\frac{e^{-t/2}}{\sqrt{2\pi t}}\frac{\sqrt{1-\rho}}{\rho}\int_{-\infty}^\infty d\lambda\, \frac{\,e^{-\frac1{2t}\lambda^2}\cosh\lambda}{\sqrt{\coth^2\lambda-\rho}}\log\frac{\sqrt{\coth^2\lambda-\rho}+\sqrt{1-\rho}}{\sqrt{\coth^2\lambda-\rho}-\sqrt{1-\rho}}.
\end{eqnarray}
We do not integrate the result because $I_{2,k}(t, \rho)$ is the expected ``bi-potential''
and the fluctuations of the  metric are obtained by differentiating it. An important aspect of \eqref{MAINintro} is that
the expression has no  $k$-dependence, except for the variable $\rho$ \eqref{rhointro}, which has the form $e^{- k D(z_1,z_2)}$ where $D(z_1,z_2)$ is the diastasis (an analog of distance-squared function for \kahler manifolds) between the points, with respect to the background metric, see \S \ref{BERGSECT}, Eq.\ \eqref{Szego}.

We consider several limits of this joint formula in \S \ref{12SECT}.  From the 
geometric viewpoint,  the most
natural scaling of the time variable is  $t_k = \epsilon_k^{-2} t$ so that the excursion
distance of the Brownian motion  in $\pcal_{N_k}$ at time $t_k$ is essentially distance
$t$ in the Mabuchi metric on $\kcalomega$. 

Similarly to \cite{ShZ, ShZ2}, we  also study the large $k$ asymptotics both in the
unscaled and scaled regimes. The scaling limit is common in related problems
in the physics of $N$ particles where one lets the number $N \to \infty$ and
the volume of the surface $V \to \infty$ such that $\frac{N}{V}$  tends to a limiting
density (see e.g. \cite{CLW} for a similar scaling in the quantum Hall effect). The natural length scale for  metrics in the Bergman space $\bcal_k$   is $\frac{1}{\sqrt{k}}$. We
consider pairs $(z_1, z_2) \in M \times M $ to be close to the diagonal if
$d(z_1, z_2) \leq \frac{\log k}{\sqrt{k}}$,  and to be  `off-diagonal' if  $d(z_1,z_2) \geq C \frac{\log k}{\sqrt{k}}$, where  $d(z_1,z_2)$ is the distance between points relative to the background
metric.
In the scaling limit we consider the asymptotics of  $I_{2, k}(t, \rho)$ for
pairs of points of the form $(z, z + \frac{u}{\sqrt{k}})$ with $|u| \leq C \log k$, in which case $\rho \simeq e^{- |u|^2}$.  The scaling asymptotics combined with the 
time scaling $t_k \to \infty$ has a limit correlation function with  a logarithmic singularity along the diagonal $z_1 = z_2$ (where
$\rho = 1$). The  variance $ (i\ddbar)_{z_1}\,(i\ddbar)_{z_2} I_{2, k}(t, \rho)$ of the \kahler
metric then has a $\delta(z_1 -z_2)$
singularity  along the diagonal. In fact, the  scaling limit  correlation function
turns out to be identical to that for zeros of random holomorphic sections determined in   \cite{ShZ} (Lemma 3.7).

   When $t  \in \R_+$  is finite and fixed, then $I_{2, k}(t, \rho)$
is smooth at $\rho = 1$ (see \eqref{SMOOTH}), and has a convergent expansion
$ (i\ddbar)_{z_1}\,(i\ddbar)_{z_2} I_{2, k}(t, \rho) \sim |z_1-z_2|^{2}$ and 
 there exist coefficients $a_n(t)$ so that 
\begin{equation}\label{TAYLOR} I_{2, k}(t, \rho) \ =\ 
\sum_{n=0}^\infty a_n(t) \rho^n.\end{equation}
Off the diagonal, $\rho \to 0$ and we can obtain the
asymptotics by Taylor expanding the amplitude of the integral \eqref{MAINintro}.
The first  term  $\frac{2t}{\rho}$ is singular as $\rho \to 0$. But in \S \ref{FIXEDTSECT}, resp. in
\S \ref{tinfty}, it is shown that the  $\frac{2 t}{\rho}$ `singularity' cancels in the
sum of the two terms. Hence the variance of the potential is exponentially
decreasing off the diagonal.

\subsection{\label{ZEROSINTRO} Comparison with \cite{ShZ,ShZ2}}

In \S \ref{tinfty} we  first let $t \to \infty$ and then let $k \to \infty$. It turns
out that in this limit,  the random metrics
we obtain are identical with random zero sets of holomorphic sections of the $k$th
power $L^k$ of the line bundle with Chern class $[\omega_0]$. As explained in \S \ref{BOUNDARY}, as $t \to \infty$ the
mass of the heat kernel concentrates on the ideal boundary of the symmetric
space, where the metrics correspond to the zero sets of holomorphic sections. On a Riemann surface,
the random metrics become normalized sums of delta functions on random point configurations with $k$ points. We verify that the pair correlation function of random \kahler
metrics in $\bcal_k$ in the limit $t \to \infty$ is given by the same formula as in \cite{ShZ} for correlations
between zeros of random sections. For large $t$ (depending on $k$), the random
metric is close to such a point configuration measure.

We now explain this similarity in more detail so that the notation and purpose of
this article are synchronized with those of \cite{ShZ, ShZ2}. The model of Gaussian
random holomorphic sections and the results are described in more detail
in \S \ref{RANDOMZEROS}.
 In those articles, the role of the area
form (in complex dimension one) is played by the zero set measure $Z_s$ of a random
section $s \in H^0(M, L^k)$, which defines a ``singular metric''. Hence it is not surprising
that there are relations between random smooth metrics and random zero sets.

  In  \cite{ShZ,ShZ2}, the zero current is given by  $Z_s = \frac{i}{\pi} \ddbar \log |s|^2$ and,
  analogously to \eqref{vc2}  the  {\it variance current\/} of zeros is defined by,
\begin{equation}\label{vc}
\var\big(Z_{s^k}\big) =  \E\big(Z_{s^k}
\boxtimes Z_{s^k}
\big) -
\E\big(Z_{s^k}\big)\boxtimes \E
\big(Z_{s^k}\big).
\end{equation} 
In \cite{ShZ}, it was shown that the bi-potential  $I_{2,k}$ of the variance\footnote{In the notations of \cite{ShZ}, $I_{2, k}$ corresponds to $4\pi^2Q_k$.} \eqref{vc} in the sense of  \eqref{varcur2}  is given by dilogarithm
\begin{equation}
\label{QN} I_{2, k}(z,w)= -
\int_0^{\rho} \frac{\log(1-s)}{s}\,ds\;.\end{equation}
In \S \ref{tinfty} we show that the $t \to \infty$ limit of the heat kernel ensemble
 gives precisely the same bi-potential, hence the same variance.

For zeros of random holomorphic sections over  a Riemann surface,  
the area statistic \eqref{AREA} with respect to $Z_s$ counts the number of
zeros of $s$ in $U$ and is denoted by $\ncal_U$ in \cite{ShZ}. 
it is shown there that
$$\var\big(\ncal_U\big) =- \int_{\d U\times
\d U}\dbar_{z_1}\dbar_{\bz_2} Q_{k} (z_1,z_2)\;.$$ and  that the number
variance for zeros has the asymptotics,
$$\var\big(\ncal_U\big) =
k^{1/2} \,\left[\nu_{1} \,\vol_{1}(\d U)
 +O(k^{-\frac 12 +\ep})\right]\;,$$ where   $\nu_{1} =
\frac{\zeta(3/2)}{8\pi^{3/2}}$. Thus, the same formula is valid in the limit $t \to \infty$
of heat kernel random metrics. Similar formulae for variances of $X_U$ and $X_{\phi}$ for random metrics can 
be derived from the explicit formula  \eqref{MAINintro} for $I_{2, k}(t, \rho)$ for any $t$ in the heat kernel measure ensembles. The details are lengthy and will be presented elsewhere.

\subsection{Asymptotic central limit theorem}

It is  shown in \cite{ShZ2} that the fluctuations of the
smooth linear statistics $X_{f}$
in the case of random zeros
tend to a Gaussian field with variance $\ncal(0, \sqrt{\kappa_1}\, \|\ddbar f \|_2)$, see \S \ref{AN} for more details. This
result holds when $t = \infty$ for heat kernel random metrics. The analogous results
for \eqref{AREA} do not seem to be known at present.

 It is very plausible that
for general times $t_k$ the smooth linear statistics $X_{f}$ with respect to the heat
kernel measure $d\mu^{t_k}_{ k}$ are also asymptotically normal, with a related variance.
Such an asymptotic central limit theorem would be a concrete measure of how
closely  heat
kernel random metrics compare to random singular metrics defined by 
point processes of random zero sets,  or to those  studied  in \cite{TBAZW,AHM,Ber,CLW},
where the fluctuations of linear statistics of eigenvalues tend to a Gaussian free field.
We plan to investigate the asymptotic normality of fluctuations of random metrics
 in future work.  These asymptotic normality results would also give a comparison
 of   heat kernel random metrics  to  Liouville random metric. The 
 fluctuations in the latter  case are of the type of
Gaussian multiplicative chaos.

\subsection{Discussion}  Heat kernel random metrics are the metrics obtained
by starting at the background metric $\omega_{\phi_I}$ and following a Brownian
motion on $\pcal_{N_k}$ for time $t$. 

The formula \eqref{MAINintro} reflects the geometry
of Brownian motion of the non-positively curved symmetric space $\pcal_{N_k}$,
which is very different from that of Euclidean space, see \S \ref{HEATSECT} 
for background.  First, due to non-isotropic nature of the Haar measure, the heat measure is concentrated along
the $SU(N_k)$-orbit of a distinguished element $\delta_{N_k}$, the half-sum of the positive roots. Second, in the radial  direction
the heat kernel measure concentrates in a kind of  annulus of radius $t$ around the $SU(N_k)$-orbit of $\delta_{N_k}$, see \S \ref{BRM}.  This is why, as  $t \to \infty$, the heat kernel measure becomes supported on the
ideal boundary $\partial_{\infty} \pcal_{N_k}$. Its $SU(N_k)$ invariance
implies that the boundary measure
is  the same as the measure on zero sets of holomorphic sections used in 
\cite{ShZ}.

As mentioned above, `heat kernel' random metrics are not like the
random metrics of Liouville quantum gravity.  On a very heuristic level,  one may
understand the difference  by thinking of $SU(N_k)$-invariance
as a discretization of invariance under the group $SDiff_{\omega_0}(M)$ of
symplectic diffeomorphisms of the background symplectic form $(M, \omega_0)$. This
is far from the invariance group of Liouville theory. It
is plausible that the  only $SDiff_{\omega_0}(M)$-invariant probability measure
on $\kcalomega$  is $\delta_{\omega_0}$.

This concentration of measure phenomenon, that heat kernel random 
metrics concentrate on the background,  is
the same phenomenon that occurs for random zeros in \cite{ShZ,ShZ2} and
for the quantum Hall point process in \cite{CLW}. Random zero sets of $N$ zeros or
random point configurations of $N$ electrons also concentrate at the
background metric as $N \to \infty$. In these cases it is customary to dilate the space to
obtain configurations of constant density. In the case of random metrics, if one dilates  small
balls of radius $\frac{1}{\sqrt{k}}$ around a point $z_0$ by the factor $\sqrt{k}$, then
the  random metrics become metrics on $\C$ and the 2-point  correlation function 
of the limit measure acquires the new term $\delta(z-w)$.

\subsection*{Acknowledgments} SK is partially supported by the Max Delbr\"uck
prize  for junior researchers at the University of Cologne, the Humboldt postdoctoral fellowship and  the grants NSh-1500.2014.2 and RFBR 15-01-04217. SZ is partially supported by NSF grant DMS-1206527 and the Bergman trust. We gratefully acknowledge support from the Simons Center for Geometry and Physics, Stony Brook University where this paper was completed.

\section{\label{BERGSECT} Bergman metrics}

We briefly review the properties of Bergman metrics, referring to \cite{PS,FKZ1}
for further background. As above, $(M, \omega,J)$ can be any compact \kahler manifold
with integral \kahler form. The Riemannian metric is $g(X, Y) = \omega(J X, Y)$.  The simplest case is that of a Riemann surface, where
a \kahler class is the same as a conformal class of metrics with fixed area. 
Instead of parametrizing metrics $g = e^u g_0$ by the Liouville field $u$ we
parametrize them by the \kahler potential $\phi$, i.e. $\omega_{\phi}= 
\omega_0 + i \ddbar \phi$, relative to the reference \kahler form $\omega_0$.

Bergman metrics of degree $k$ are special \kahler metrics induced by holomorphic embeddings 
$$\iota_{s}(z)  = [s_1, \dots, s_{N_k}] :M \to \CP^{N_k-1}$$
of $M$ into complex projective space. Here, $\{s_j\}$ is a basis of the space $H^0(M, L^k)$ of holomorphic sections of powers $L^k$ of
an ample line bundle $L \to M$ with first Chern class $c_1(L) = [\omega_0]$. Also $N_k = \dim H^0(M, L^k)$.  Given a reference basis $\{s_j\}$ one obtains all others by applying an element $A \in GL(N_k, \C)$ to it ${s^A}_j=\sum A_{jl}s_l$ and induces the embedding $$\iota_{s_A}: M \to \CP^{N_k-1},
\;\; \iota_{s_A} = A \circ \iota_{s}. $$The associated
Bergman metric is then,
\begin{equation} \label{BergMetDef} \iota_{s_A}^* \omega_{FS} = \frac{1}{k} i \ddbar \log \sum_{j = 1}^{N_k} |{s^A}_j(z)|^2. \end{equation}
Since $U(N_k)$ is the isometry group of $\omega_{FS}$, the space of metrics is the quotient symmetric
space $\pcal_{N_k} = GL(N_k,\C)/U(N_k)$. With no loss of generality one may restrict to $SL(N_k, \C)$ and obtain the
quotient $SL(N_k, \C)/SU(N_k)$. 

 We choose a basis of sections $\{s_i(z)\}=\{s_1(z),...,s_{N_k}(z)\}$
 of $H^0(M, L^k)$ which is orthonormal with respect to the reference (background) metric $h_0^k$ on $L^k$ and the corresponding \kahler metric $\omega_0=-\frac1k i \ddbar \log h_0^k$ on $M$
\begin{equation}
\label{ortho}
\frac1V\int_M \bs_i(z)s_j(z)h_0^k\,\frac{\omega_0^n}{n!}=\delta_{ij},
\end{equation}
where $n = \dim M$. The Bergman kernel of the background metric is the kernel of the orthogonal
projection onto $H^0(M, L^k)$ with respect to the inner product above, and is
given by
\begin{equation} \label{BKer}B_k(z_1, z_2) = \sum_{j = 1}^{N_k} s_j(z_1) \bs_j(z_2) \end{equation}
Given a positive Hermitian matrix $P = P_{ij}$ the associated Bergman metric is,
\begin{equation}
\label{bergm}
\omega_{a\bb}(z)=\frac1k\p_a\bp_{\bb}\log \bs_i(z)P_{ij}s_j(z).
\end{equation}
In terms of $A \in GL(N_k, \C)$ above,  $P=A^{\dagger}A$. We introduce the Bergman potential as follows
\begin{equation}
\label{KP}
\phi_P=\frac1k\log \bs_i(z)P_{ij}s_j(z)=\frac{1}{k} \log | \langle e^{\Lambda}
U s(z), U s(z) \rangle|^2.
\end{equation}

 A key property is that  $\kcalomega  = \overline{\bigcup_k \bcal_k}$, i.e. the full space of metrics 
in a fixed \kahler class  is the closure
of the set of Bergman metrics. Hence $\kcalomega$ is well approximated by $\bcal_k$ for large $k$, and there are now
many results showing that it is well approximated in  much stronger geometric ways.
This approximation problem was raised by S.~T.~Yau in \cite{Y}, see \cite{PS} for background.

\subsection{ Berezin kernel}
     The key invariant is the Berezin kernel \eqref{rhointro}, given in 
the above notation by

\begin{equation}
\label{Szego} \rho =
\frac{|\langle s(z_1),s(z_2)\rangle|^2}{|s(z_1)|^2|s(z_2)|^2},
\end{equation}
or  in terms of the Bergman kernel
\begin{equation}\label{PN} \rho = P^2_k(z_1,z_2):=
\frac{|B_k(z_1,z_2)|^2}{B_k(z_1,z_1)B_k(z_2,z_2)}\;.\end{equation}

\subsection{Matrix-metric correspondence} The matrix-metric correspondence Eq.\ \eqref{bergm} uses a choice of basis $\{s_j\}$ of $H^0(M, L^k)$. Any natural measure on $\bcal_k$ must be independent of the choice of this basis. We pause to  describe such natural
measures. 

Any \kahler metric $\omega = \omega_0 + i \ddbar \phi$ in $\kcalomega$ induces
an inner product ${\rm Hilb}_k(\phi)$ on $H^0(M, L^k)$ by the rule
$$\langle s_1, s_2 \rangle_{{\rm Hilb}_k(\phi)} = \int_M \bs_1(\bz) s_2(z) h^k \frac{\omega^n}{n!}.$$
Given a background inner product $G_0 = {\rm Hilb}_k(\phi_0)$, any
other inner product has the form $\langle s_1, s_2 \rangle_G = \langle P_G s_1, s_2 \rangle_{G_0}$ where $P_G$ is a positive Hermitian operator on $H^0(M, L^k)$ with
respect to $G_0$. It has a well-defined  polar decomposition $e^{\Lambda} U$
where $U \in U(G_0)$ is unitary with respect to $G_0$.  Its eigenvalues are encoded by the diagonal matrix $\Lambda_G$  and its  eigenvectors are encoded by $U$.

In making calculations, we need to parametrize such positive Hermitian operators
by positive Hermitian matrices, which requires a choice of a $G_0$-orthonormal
basis of $H^0(M, L^k)$. Any measure intrinsically defined on the space of positive
Hermitian operators will be independent of the choice of basis. Haar measure and
the heat kernel are examples of such measures.

\section{\label{HEATSECT} Heat kernel}

In this section we review the heat kernel on $\pcal_{N}$\footnote{In this section we adopt shorthand notation $N=N_k$.}. Bergman metrics
are unchanged if the positive Hermitian matrix $P$ is multiplied by a scalar, so 
we may normalize $P$ so that $\det P =1$. Then $\pcal_{N} = G/K$ where  $G = SL(N, \C)$
and $K = SU(N)$.
We denote by $\mathfrak k$
the Lie algebra of the maximal compact subgroup $K \subset G$ and let $\mathfrak g = \mathfrak k \oplus \mathfrak p$.
Let $\mathfrak a$ be a maximal abelian subspace of $\mathfrak p$ and let $\ell = \dim \mathfrak a$. The set of positive roots is denoted
by $R_+$.   
The roots are $e_i -
e_j$, and the positive roots satisfy $i < j$ and have multiplicity one.  For $SL(N,\C)/SU(N)$ the half sum of the positive roots is the element
$\delta_N = (-\frac{N-1}2,  -\frac{N-3}2
\dots,\frac{ N-1}2)$. For background, see \cite{K,H}.

 We refer to the matrix decomposition $P=U^\dagger e^\Lambda U$ for $\Lambda={\rm diag}(\lambda_1,\ldots,\lambda_N)$, and $U \in U(N)$ as
 `polar coordinates' on $\pcal_N$, 
where real numbers $\lambda_j\in(-\infty,+\infty)$ correspond to the Cartan elements of $SL(N,\C)$.

 The CK (Cartan-Killing) metric is  given by
\begin{equation} \label{CK}
ds^2={\rm Tr}(P^{-1}dP)^2
\end{equation}
for $P\in GL(N,\mathbb C)/U(N)$. This metric is bi-invariant  under the action of $GL(N,\mathbb C)$.

   The associated volume form  $dV$  on the symmetric space $SL(N, \C)/SU(N)$ of positive Hermitian matrices with  $\det P = 1$
 is the bi-invariant Haar measure,
\begin{equation}
\label{eq3}
dV=\delta\left(\sum_{j=1}^N\lambda_ j\right)\Delta^{ 2  }(e^\lambda)\prod_{j =1}^Nd\lambda_ j\cdot\frac{[dU]}{[dU_{U(1)^N}]},
\end{equation}
where $[dU]$ is the standard Haar measure on unitary group.

\subsection{Heat kernel measure on $SL(N,\C)/SU(N)$}

Following Gangolli \cite{gagnolli} (Proposition 3.2; see also \cite{AO}, section
2), the heat kernel on $SL(N,\mathbb C)/SU(N)$
with respect to the standard CK (Cartan-Killing) metric  is given in `polar coordinates'
$(\lambda, U)$ on $\pcal_{N}$ by
\begin{equation}
\label{eq1}
d\mu_t=g_t(\lambda) dV= C(t,N)\frac{\Delta(\lambda)}{\Delta(e^\lambda)}e^{-\frac1{4t}\sum_{j=1}^{N}\lambda_j ^2} dV.
\end{equation}
Here,  $\Delta(\lambda)=\prod_{i<j}(\lambda_ j-\lambda_i)$ is the standard Vandermonde determinant.

The normalization constant $C(t,N)$ in \eqref{eq1} is fixed by the condition that $\mu_t$ is the probability measure $\int d\mu_t=1$,
\begin{equation}
\label{C}
C(t,N)=\frac{\sqrt N}{2\pi(\sqrt{4\pi t})^{N^2-1}}e^{-\frac t{12}N(N^2-1)}.
\end{equation}
In deriving this we use the volume of the unitary group
$$
{\rm Vol}\,U(N)=(2\pi)^{N(N+1)/2}/\prod_{j=1}^Nj!,
$$
see e.g. \cite{M}.
The factor $e^{-\frac t{12}N(N^2-1)}$ is $e^{- t ||\delta_N||^2}$ and arises because
$||\delta_N||^2 $ is the bottom of the spectrum of the Laplacian.
Putting \eqref{eq3} and \eqref{C} together, we get the following expression 
\begin{eqnarray}\label{mutdef}
d\mu_t=\frac{\sqrt N}{2\pi(\sqrt{4\pi t})^{N^2-1}}e^{-\frac t{12}N(N^2-1)}\delta\left(\sum_{j=1}^N\lambda_j\right)\Delta(\lambda)
\Delta(e^\lambda)e^{-\frac1{4t}\sum_{j=1}^{N}\lambda_j^2}
\prod_{j=1}^Nd\lambda_ j\cdot\frac{[dU]}{[dU_{U(1)^N}]}
\end{eqnarray}
for the heat kernel measure on $SL(N,\mathbb C)/SU(N)$ with respect to the
CK metric.

\subsection{\label{BRM} Geometry of the heat kernel and Brownian motion}

In \cite{AS} it is proved that the mass of the heat kernel concentrates along the exponential image of the  $U(N)$-orbit of the $\delta_N$-axis in a small
annulus centered at $2 |\delta_N| t $.

If we write $H = \rm{diag}(\lambda)$, then the Gaussian factor $t^{- (N^2-1)/2} e^{-\frac{||H||^2}{4t}}$ is similar to the heat kernel of Euclidean space. But this Gaussian
factor must compete with the exponential volume growth factor $
\Delta(e^\lambda)$ and the factor $e^{-\frac t{12}N(N^2-1)}$ due to the existence
of a spectral gap for $\Delta$. The well-known factor $\Delta(\lambda)$ pushes
the eigenvalues of $\log P$ apart. The factor  $J(H)$ is bounded by
$e^{2 \langle \delta_N, \vec \lambda \rangle}$ and a simplified expression
for the heat kernel is $e^{- t |\delta_N|^2 + \langle \lambda, \rho_N \rangle
- \frac{|\lambda|^2}{2t} }.$ The maximum of the exponent occurs when $\vec \lambda = 2t \delta_N$.

Following \cite{AS}, let $\gamma(t)$ be a positive function with $\sqrt t \gamma(t) \to \infty$ as $t \to \infty$, and let
$R(t)$ be a positive function such that $R(t)/\sqrt t \to \infty$. Consider the annulus
$$A(2 |\delta_N| t - R(t), 2 |\delta_N| t + R(t)): = \{H: 2 |\delta_N| t - R(t) \leq |H| \leq 2 |\delta_N| t + R(t)\} \subset \mathfrak a $$
and consider the solid cone 
$$\Gamma(t) = \mbox{solid cone around the}\; \delta_N\; \mbox{axis of angle }\; \gamma(t), $$
and let
$$\Omega(t) = A(2 |\delta_N| t - R(t), 2 |\rho| t + R(t)) \cap \Gamma(t). $$
Then, according to Theorem 1 of \cite{AS},
\begin{equation} \int_{U(N) \exp \Omega(t) U(N) } d\mu_t \to 1, \;\; t \to \infty. \end{equation}

A Brownian motion proof of this result is given in \cite{Bab}. It shows that
as $t \to \infty$ the mass of $\mu_t$ moves off to a component of the ideal boundary (at
infinity) of $\pcal_{N}$. In \S \ref{BOUNDARY} we discuss this boundary.

\subsection{\label{SCALINGSECT} Scaling and dilation}

In the large $k$ limit, the symmetric space metric on $\bcal_k$, when properly
scaled, tends to the Mabuchi metric $g_M$ on $\kcalomega$ (see \cite{CS}). The Mabuchi
distance function is induced by the Riemannian metric on $\kcalomega$ 
defined by $||\delta \phi||^2_{\phi_0} = \int_M (\delta \phi)^2 \omega_{\phi}^n/n!
$ where $\omega_{\phi} = \omega_0 + i \ddbar \phi$. We
refer to \cite{PS} for background.
For all $k$, $\bcal_k \subset \kcalomega$. If we  rescale
the CK  metric  $g_{CK,k}$ \eqref{CK}  as $ g_k =  \epsilon_k^2 g_{CK, k}$, with
  $\epsilon_k = k^{-1} N_k^{-1/2}$, then $g_k \to g_M$ on $T \bcal_k$.
Thus, a ball of radius one with respect to the usual CK metric $g_{CK,k}$ has radius
approximately $\epsilon_k$ with respect to the Mabuchi distance. 
It is obviously desirable to consider the heat kernel measures for this rescaled sequence
of metrics.

If we rescale the CK metric to $g_k =\epsilon_k^2 g_{CK, k}$ the corresponding
Laplacian scales as $\Delta_{g_k} \to \epsilon_k^{-2} \Delta_{g_{CK, k}}$. It follows
that the heat operator scales as
$$\exp t \Delta_{g_k} = \exp t \epsilon_k^{-2} \Delta_{g_{CK, k}}. $$
In effect, it is only the time that is rescaled and the rescaled heat kernel
is $p_k(\epsilon_k^{-2} t, I, P). $

\section{\label{12SECT} One and two point correlation functions of random metrics}

In this section, we calculate the one and two point functions of the
random  \kahler potential; in  the introduction, the latter was stated to be  \eqref{MAINintro}. We use the notation $\E = \E_k$ for the  expectation, which at
the beginning  could be with respect
to any $U(N_k)$-invariant measure and then specializes to the heat kernel
measures. For simplicity of notation we often abbreviate $N_k$ by $N$
and drop the explicit $k$-dependence.

As mentioned in the introduction, the one and two-point functions are the data required
to study the mean and variance of the area random variables $X_U$. Evidently,
$$\E_k X_U = \int_U \E_k\, \omega, \;\; \Var(X_U) = \int_{U\times U} \E_k \left[\omega(z_1) \omega(z_2) \right] - \int_{U\times U}\, \E_k [\omega(z_1)] \E_k\, [\omega(z_2)] . $$
The integrands are the one- and two-point correlation functions.

\subsection{One point function of the K\"ahler potential}

The argument here follows \cite{FKZ2}. Using the integral representation of the logarithm 
\begin{equation}
\label{log}
\log\alpha=\frac1\tau+\gamma-\int_0^\infty x^{\tau-1}e^{-\alpha x}dx+\mathcal O(\tau), 
\end{equation}
where $\gamma$ is the Euler constant, we can rewrite the expectation value of the \kahler potential $\phi_P$ \eqref{KP} relative to the  \kahler potential $\phi_I$ of the background Bergman metric as
\begin{align}
\label{1point}\nonumber
\E_k[\phi_P(z)-\phi_I(z)]&=\E_k\left[\frac1k\log\frac{ \bar s(z) U^\dagger e^\Lambda U s(z)}{|s(z)|^2}\right]=\\
&=\frac1k\lim_{\tau\rightarrow0}\;\;\frac1\tau+\gamma-\int_0^\infty x^{\tau-1}dx\int_{\pcal_{N}}
e^{-{\rm Tr}\,e^\Lambda U\Psi U^\dagger}d\mu_t,
\end{align}
where introduced the matrix $\Psi_{jl}=xs_j(z) \bar s_l(z)/|s(z)|^2$. The integration over the unitary group can be carried out  using Harish-Chandra-Itzykson-Zuber formula, see e.\ g.\ \cite{M,ZZ} for background. Namely, for any two Hermitian matrices $A$ and $B$ with eigenvalues $a_j$ and $b_j$
\begin{equation}
\label{HCIZ}
\int_{U(N)}\frac{[dU]}{{\rm Vol}\, U(N)} \exp\left(\mu{\rm Tr} AUBU^\dagger\right)=\left(\prod_{p=1}^{N-1}p!\right)
\mu^{-N(N-1)/2}\frac{\det\left(e^{\mu a_j b_l}\right)_{1\leq j,l\leq N}}{\Delta(a)\Delta(b)}.
\end{equation}
It is not hard to check that this expression is well defined even if some of the eigenvalues coincide. This is the case for the matrix $\Psi$ which has $N-1$ zero eigenvalues and one non-zero eigenvalue equal to $x$. Hence the integral on the right hand side of \eqref{1point} is $z$-independent and we immediately conclude that the expectation value of the Bergman metric is equal to the background Bergman metric
$$
\E_k[\omega_{a\bar b}]=\omega_{\phi_I, a\bar b}
$$
In fact, this is true for any eigenvalue-type measure \cite{FKZ2}, since the HCIZ integral depends on eigenvalues only. Note, that so far no assumptions on $k$ have been made in this calculation. Considering now the limit of  $k$ large, we invoke the Bergman kernel expansion \cite{FKZ2} to show that the background Bergman metric $\omega_{\phi_I}$ tends to the reference \kahler metric, 
$$
\omega_{\phi_I, a\bar b}= \omega_{0\, a\bar b}+\mathcal O(1/k).
$$
Now we would like to consider the variance of $X_U$ from its mean $\omega_0(U)$. 

\subsection{\label{twopoint}The two point function} In this section we prove
the formula \eqref{MAINintro} for the two point function, and then discuss its
asymptotics in various regimes. In the terminology of \cite{ShZ} we are finding
a `bi-potential' for the variance. Although the calculations of this bi-potential are
completely different from the case of random holomorphic sections in \cite{ShZ},
the final formulae are somewhat similar and when $t \to \infty$ they are identical.

 Recall that the two-point function \eqref{final2p}  of the \kahler potential
 is the sum of the background term $\phi_I(z) \phi_I(w)$ plus the variance term
 $I_{2,k}(\rho)$. Instead of using the formula \eqref{I2} for this term, we take
 the approach of writing
\begin{align}
\label{2point}\nonumber
K_{2, k}(z,w) - \phi_I(z) \phi_I(w) & =\frac1{k^2}\E_k\left[\log\frac{ \bar s (z_1)U^\dagger e^\Lambda U s(z_1)}{|s(z_1)|^2}\log \frac{\bar s(z_2) U^\dagger e^\Lambda U s(z_2)}{|s(z_2)|^2}\right]\\ &=\frac1{k^2}\lim_{\tau_1,\tau_2\rightarrow0}\bigl(I_{2,k}(t,\rho,\tau_1,\tau_2)+\;\rho\mbox{-}{\rm independent\,terms}\bigr) \end{align}
where we do not write down the $\rho$-independent terms, since  ultimately we are interested in the dependence of the correlation function on coordinates, which enter only through $\rho(z_1,z_2)$. Here
\begin{align}\label{2point1} I_{2,k} (t,\rho,\tau_1,\tau_2)=\iint_0^\infty x_1^{\tau_1-1}x_2^{\tau_2-1}dx_1dx_2\int_{\pcal_{N}} e^{-{\rm Tr}\,e^\Lambda U\Phi U^\dagger}d\mu_t.
\end{align}
and  we introduced the matrix $\Phi_{jl}=x_1\frac{s_j(z_1)\bar s_l(z_1)}{|s(z_1)|^2}+ x_2\frac{s_j(z_2)\bar s_l(z_2)}{|s(z_2)|^2}$. It has rank 2 with two non-zero eigenvalues given by
$$
\phi_{1,2}=\frac12\left(x_1+ x_2\pm\sqrt{(1-\rho)(x_1-x_2)^2+\rho(x_1+x_2)^2}\right).
$$
Applying the HCIZ formula (\ref{HCIZ}) to the unitary integration in Eq.\ (\ref{2point1}) we obtain 
$$
(-1)^{N(N-1)/2}\frac{N!(N-1)!}{(\phi_1\phi_2)^{N-2}(\phi_1-\phi_2)}\frac{e^{-\phi_1e^{\lambda_1}-\phi_2e^{\lambda_2}}
}{\prod_{j=2}^N(e^{\lambda_1}-e^{\lambda_j})\prod_{l=3}^{N}(e^{\lambda_2}-e^{\lambda_l})}$$
where we used the fact that the integration measure is symmetric in eigenvalues $\lambda$'s.

Now we use the explicit form of the heat kernel measures $\mu_t$ \eqref{mutdef}.
The integral $I_{2,k} (t,\rho,\tau_1,\tau_2)$ in Eq.(\ref{2point}) with this eigenvalue
measure can be written as\begin{eqnarray}
\nonumber
&&I_{2,k}(t,\rho,\tau_1,\tau_2)=(-1)^{N(N-1)/2}C(t,N)\frac{{\rm Vol}\, U(N)}{(2\pi)^N}N!(N-1)!\iint_0^\infty \frac{x_1^{\tau_1-1}x_2^{\tau_2-1}dx_1 dx_2}{(\phi_1\phi_2)^{N-2}(\phi_1-\phi_2)}\\&&\int_{-\infty}^\infty dy \int \prod_{j=1}^Nd\lambda_j\;\Delta(\lambda)\Delta_{12}(e^\lambda)\exp\left(-\frac1{4t}\sum_{j=1}^{N}\lambda_j^2+iy \sum_{j=1}^N\lambda_j-\phi_1e^{\lambda_1}-\phi_2e^{\lambda_2}\right),\nonumber
\end{eqnarray}
where we defined the partial Vandermonde determinant $\Delta_{12}(e^\lambda)=\prod_{3\leq j<l\leq N}(e^{\lambda_j}-e^{\lambda_l})$, which excludes the first two eigenvalues $e^{\lambda_1}$ and $e^{\lambda_2}$. The $y$-integration enforces the delta-function constraint in the measure \eqref{mutdef}.  

Using antisymmetry of $\Delta(\lambda)$ under exchange of two eigenvalues, $\Delta_{12}(e^\lambda)$ can be replaced by $(-1)^{1+N(N-1)/2}(N-2)!\,e^{\sum_{l=3}^N(l-3)\lambda_l}$ inside the integral, which leads to the further simplification
\begin{eqnarray}
\nonumber\label{int19}
&&I_{2,k}(t,\rho,\tau_1,\tau_2)=-\frac{{\rm Vol}\, U(N)}{(2\pi)^N}C(t,N)N!(N-1)!(N-2)!\iint_0^\infty\frac{x_1^{\tau_1-1}x_2^{\tau_2-1}dx_1dx_2}{(\phi_1\phi_2)^{N-2}(\phi_1-\phi_2)}\cdot\\&&\int_{-\infty}^{\infty} dy \int \prod_{j=1}^Nd\lambda_j\,\label{int21}
\Delta(\lambda)\,e^{\sum_{j=1}^{N}\bigl(-\frac1{4t}\lambda_j^2+iy \lambda_j\bigr)+\sum_{l=3}^N(l-3)\lambda_l-\phi_1e^{\lambda_1}
-\phi_2e^{\lambda_2}},
\end{eqnarray}
Thus after the HCIZ integration we got rid
of most difficult factor $\Delta(e^{\lambda})$ and left with a Gaussian integral with
a polynomial amplitude, except for the
terms $\phi_1 e^{\lambda_1} + \phi_2 e^{\lambda_2}$  in the exponent. Note that the $k$-dependence is entirely in
the variable $\rho$ inside $\phi_1, \phi_2$.

The next step is to calculate the integrals in the second line of \eqref{int19},
\begin{equation} \ical_{2, k} (t, \phi_1, \phi_2) := \int_{-\infty}^{\infty} dy \int \prod_{j=1}^Nd\lambda_j\,\label{int21}
\Delta(\lambda)\,e^{\sum_{j=1}^{N}\bigl(-\frac1{4t}\lambda_j^2+iy \lambda_j\bigr)+\sum_{l=3}^N(l-3)\lambda_l-\phi_1e^{\lambda_1}
-\phi_2e^{\lambda_2}}. \end{equation} 
Our strategy is to Taylor-expand the exponent of \eqref{int21}  in powers of $e^\lambda$, 
$$
e^{-\phi_1e^{\lambda_1}-\phi_2e^{\lambda_2}}=\sum_{m,n=0}^{\infty}\frac{(-\phi_1)^n(-\phi_2)^m}{n!m!}e^{n\lambda_1+m\lambda_2},
$$ then perform gaussian integration in $\lambda$ term-by-term,  and finally re-sum the resulting series, i.\ e.\ un-do the Taylor expansion. We use the identity
\begin{equation}
\label{iden}
\prod_{j=1}^N\left(\int_{-\infty}^{\infty}d\lambda_j\right)\, \Delta(\lambda)\,e^{\sum_{j=1}^N(-\frac1{4t}\lambda_j^2+\mu_j\lambda_j)}
=(2\pi)^{N/2}(2t)^{N^2/2}\Delta(\mu)\,e^{t\sum_{j=1}^N\mu_j^2}
\end{equation}
to compute the eigenvalue integral and get
\begin{align}
\nonumber \label{Ical}
\ical_{2, k} (t, \phi_1, \phi_2)&= \sum_{n,m=0}^{\infty}\frac{(n-m)}{n!m!}(-\phi_1)^n(-\phi_2)^m
\prod_{l=0}^{N-3}(n-l)(m-l)\\&\cdot
\int_{-\infty}^{\infty} dy\,e^{t(n+iy)^2+t(m+iy)^2+t\sum_{l=0}^{N-3}(l+iy)^2}.
\end{align}
Due to the factor $\prod_{l=0}^{N-3}(n-l)(m-l)$,  all terms with $m,n<N-2$ have coefficient zero, so we can shift summation indices $n\rightarrow n-(N-2)$, $m\rightarrow m-(N-2)$. Integrating  over $y$ in \eqref{Ical} and plugging the result back to \eqref{int19} we obtain
\begin{eqnarray}
\nonumber
&&I_{2, k}(t,\rho,\tau_1,\tau_2)=-e^{-t(N-1)^2/N}
\iint_0^\infty\frac{x_1^{\tau_1-1}x_2^{\tau_2-1}dx_1dx_2}{\phi_1-\phi_2}\cdot
\\&&(\phi_1\partial_{\phi_1}-\phi_2\partial_{\phi_2})
\sum_{n,m=0}^{\infty}\frac{(-\phi_1)^n(-\phi_2)^m}{n!m!}\,e^{\frac t2\frac{N-2}N(n+m)^2+\frac t2(n-m)^2+t\frac{(N-1)(N-2)}N(n+m)}.\label{rel1}
\end{eqnarray}
Now we can re-sum the series using the identity
\begin{eqnarray}\nonumber
\frac1{2\pi t}\int_{-\infty}^{\infty}d\lambda_1d\lambda_2\,e^{-\frac1{2t}(\lambda_1^2+\lambda_2^2)-\phi_1
e^{a\lambda_1+\lambda_2}-\phi_2e^{a\lambda_1-\lambda_2}}=
\sum_{n,m=0}^{\infty}\frac{(-\phi_1)^n(-\phi_2)^m}{n!m!}\,
e^{\frac t2a^2(n+m)^2+\frac t2(n-m)^2}
\end{eqnarray}
Replacing the series in (\ref{rel1}) by the integral, and changing variables
 $x_1\rightarrow e^{-a\lambda_1-t\frac{(N-1)(N-2)}N}x_1$, $x_2\to e^{-a\lambda_1-t\frac{(N-1)(N-2)}N}x_2$ makes it possible to carry out the gaussian integration in $\lambda_1$,
 giving
\begin{eqnarray}
\nonumber
I_{2, k}(t,\rho,\tau_1,\tau_2)=-\frac{e^{-t/2}}{\sqrt{2\pi t}}
\iint_0^\infty\frac{x_1^{\tau_1-1}x_2^{\tau_2-1}dx_1dx_2}{\phi_1-\phi_2}(\phi_1\partial_{\phi_1}-\phi_2\partial_{\phi_2})\int_{-\infty}^{\infty}d\lambda\, e^{-\frac1{2t}\lambda^2-\phi_1e^\lambda-\phi_2e^{-\lambda}}
\end{eqnarray}
Interestingly, the factor $e^{- t |\delta_N|^2}$ coming from the spectral gap has
now disappeared from the formula.
It follows from the integral representation Eq. \eqref{log}, that the singular in $\tau_{1,2}$ terms in \eqref{2point} are $\rho$-independent. Therefore after taking the derivative of $I_{2,k}(t,\rho,\tau_1,\tau_2)$ with respect to $\rho$, we can set $\tau_1=\tau_2=0$. Since the measures $\mu_t$ depend on $t$ we henceforth denote the
corresponding term in the two point function by $I_{2,k}(t, \rho)$.
Using
$$
\partial_\rho\phi_{1,2}=\pm\frac{x_1x_2}{\phi_1-\phi_2},\,\,\,\,\,\,\,\,\partial_\rho\frac1{\phi_1-\phi_2}=-\frac{2x_1x_2}{(\phi_1-\phi_2)^3}
$$
we get 
\begin{align}
\nonumber
\p_\rho I_{2,k}(t, \rho)=&
\frac{2e^{-t/2}}{\sqrt{2\pi t}}\iint_0^\infty dx_1 dx_2\int_{-\infty}^{\infty}d\lambda \,e^{-\frac1{2t}\lambda^2-\phi_1e^\lambda-\phi_2e^{-\lambda}}\\&\cdot\nonumber
\frac{\sinh{\lambda}}{(\phi_1-\phi_2)^3}\bigl(\phi_1+\phi_2+(\phi_1-\phi_2)(\phi_1e^\lambda-\phi_2e^{-\lambda})\bigr)
\end{align}
Now we introduce new coordinates $(r,\theta)$ as
$$
r\cos\theta=\sqrt{\rho}\,(x_1+x_2),\,\,\,\,\,\,r\sin\theta=\sqrt{1-\rho}\,(x_2-x_1),
$$
with the range
$$r\in[0,\infty),\quad
\theta\in[-\alpha,+\alpha] \quad{\rm where}\quad \cos\alpha=\sqrt\rho.
$$
In terms of $r$ and $\theta$ the integrals can be written as
\begin{align}
\nonumber
\p_\rho I_{2,k}(t,\rho)=&-\frac{e^{-t/2}}{\sqrt{2\pi t}}\frac1{\sqrt{\rho(1-\rho)}}\int_0^\infty rdr\int_{-\alpha}^\alpha d\theta \int_{-\infty}^\infty d\lambda\,e^{-\frac1{2t}\lambda^2-r\bigl(\frac{\cos\theta}{\sqrt\rho}\cosh\lambda+\sinh\lambda\bigr)}\\&\nonumber\cdot\left[\frac{\cos\theta}{r^2\sqrt\rho}+\frac{1}{2r}
\left(\frac{\cos\theta+\sqrt\rho}{\sqrt\rho}e^\lambda-\frac{\cos\theta-\sqrt\rho}{\sqrt\rho}e^{-\lambda}\right)
\right]\sinh\lambda.
\end{align}
Introducing new variable $x\sqrt\rho=\cos\theta$ and rearranging exponents, we get
\begin{eqnarray}
\nonumber
&&\p_\rho I_{2,k}(t, \rho)=-\frac{e^{-t/2}}{\sqrt{2\pi t}}\frac1{\sqrt{1-\rho}}\int_0^\infty rdr\int_1^{1/\sqrt\rho} \frac{dx}{\sqrt{1-\rho x^2}} \int_{-\infty}^\infty d\lambda\,e^{-\frac1{2t}\lambda^2+\lambda}\\&&\nonumber
e^{-rx\cosh\lambda}\left[
\frac x{r^2}\left(e^{-r\sinh\lambda}-e^{r\sinh\lambda}\right)+\frac{\sinh\lambda}r\left(
(x+1)e^{-r\sinh\lambda}+
(x-1)e^{r\sinh\lambda}
\right)
\right].\nonumber
\end{eqnarray}
Integrating over $r$ we obtain
\begin{align}
\nonumber
\p_\rho I_{2,k}(t,\rho)=&-\frac{e^{-t/2}}{\sqrt{2\pi t}}\frac1{\sqrt{1-\rho}}\int_1^{1/\sqrt\rho} \frac{dx}{\sqrt{1-\rho x^2}} \int_{-\infty}^\infty d\lambda\,e^{-\frac1{2t}\lambda^2+\lambda}\\&\cdot\nonumber
\left[x\log\left(\frac{x\cosh\lambda-\sinh\lambda}{ x\cosh\lambda+\sinh\lambda}\right)+2\sinh\lambda\,
\frac{x^2\cosh\lambda-\sinh\lambda}{x^2\cosh^2\lambda-\sinh^2\lambda}\right]\nonumber.
\end{align}
After integrating the log term by parts, we can perform the $x$ integration
\begin{align}\label{MAIN1}\nonumber
&\p_\rho I_{2,k}(t, \rho)=\frac{2t}\rho-\frac{2e^{-t/2}}{\sqrt{2\pi t}}\frac{\sqrt{1-\rho}}{\rho}\int_{-\infty}^\infty d\lambda\,
\int_1^{1/\sqrt\rho} \frac{dx}{\sqrt{1-\rho x^2}}\cdot\frac{e^{-\frac1{2t}\lambda^2}\cosh\lambda\sinh^2\lambda}{x^2\cosh^2\lambda-\sinh^2\lambda}\\&
=\frac{2t}\rho-
\frac{e^{-t/2}}{\sqrt{2\pi t}}\frac{\sqrt{1-\rho}}{\rho}\int_{-\infty}^\infty d\lambda\,e^{-\frac1{2t}\lambda^2}\frac{\cosh\lambda}{\sqrt{\coth^2\lambda-\rho}}\log\frac{\sqrt{\coth^2\lambda-\rho}+\sqrt{1-\rho}}{\sqrt{\coth^2\lambda-\rho}-\sqrt{1-\rho}}.
\end{align}
This completes the calculation of \eqref{MAINintro}.

  An important application of these asymptotics is to calculate the variances of
  the linear statistics $X_U $ \eqref{AREA} and its smooth analogue $X_{f}.$
  This can be done precisely as in Section 4 of \cite{ShZ} and Section 3 of \cite{ShZ2},
  but substituting the formula for $I_{2, k}(t, \rho)$ for $Q_k$. The details are rather
  lengthy and will be presented elsewhere.

The remainder of the article
is devoted to the asymptotics of $I_{2,k} (t, \rho)$ as $k \to \infty$ with various
relations between $t$ and $k$.  To understand the various regimes, it should
be recalled that
 the metric rescaling in \S \ref{SCALINGSECT} is necessary to  ensure
that balls in the Cartan-Killing metric on $\bcal_k$ maintain their size with respect
to the limiting Mabuchi metric on $\kcalomega$.  Without rescaling,  the Brownian motion relative to the CK metric  is probing metrics only at distance $d_k =
\epsilon_k^2 t$ from the initial background metric with respect to the
limiting metric.

 In all of the regimes, the key to finding the scaling asymptotics is to work out
 the behavior of
 
\begin{equation} \label{AMP} A(t, \rho): = \frac{1}{\sqrt{\coth^2 t-\rho}}\log\frac{\sqrt{\coth^2 t-\rho}+\sqrt{1-\rho}}{\sqrt{\coth^2 t-\rho}-\sqrt{1-\rho}} \end{equation}
as $k \to \infty$, where $t$ may depend on $k$. As will be seen below, the factors
of $(\sqrt{2 \pi t})^{-1} e^{-t/2}$ in front of the integral are always cancelled, leaving
the prefactor $\frac{\sqrt{1 - \rho}}{\rho}$. By \eqref{BERSCALING}, for pairs $(z, z + \frac{u}{\sqrt{k}})$, we have
$\rho \to e^{-|u|^2}$, and $A_k(t, \rho)$ has a limit, which depends on whether or
not we also send $t \to \infty$.
  

\subsection{\label{FIXEDTSECT} The limit as $k \to \infty$ for fixed $t$} 
This regime corresponds to letting the Brownian motion with respect to
the Cartan-Killing metric evolve for a time $t$, and as discussed in \S \ref{SCALINGSECT} the   ball of radius $t$ in $g_{CK, k}$ metric is shrinking in size
with respect to the Mabuchi metric and has $d_k$-radius equal to $\epsilon_k t.$

The asymptotics in this regime could in principle be derived from   the formula  \eqref{I2} of \cite{FKZ2}. But more explicitly,
we note that $\rho \to 0$ as $k \to \infty$ off the diagonal. Expanding at small $\rho$, we get
\begin{align}\label{SMOOTH}
\frac1{\sqrt{2\pi t}}e^{-t/2}\int_{-\infty}^\infty d\lambda\,e^{-\frac1{2t}\lambda^2}\frac{\cosh\lambda}{\sqrt{\coth^2\lambda-\rho}}\log\frac{\sqrt{\coth^2\lambda-\rho}+\sqrt{1-\rho}}{\sqrt{\coth^2\lambda-\rho}-\sqrt{1-\rho}}=2t+\mathcal O(\rho),
\end{align}
and the first term here cancels the first term in \eqref{MAIN1}. Thus in the regime when  we hold $(z_1,z_2)$ fixed then $\rho \to 0$, we get
$$  I_{2, k}(t, \rho) \simeq a_0(t) + a_1(t) \rho + \cdots,$$
where $a_0,a_1,...$ are constants independent of $\rho$. To obtain the 2-point correlation
function of the \kahler metric, we then take four more derivatives. The 
constant $a_0$ does not contribute to the answer and we see that the two point correlation function
is the free background term $\omega_{\phi_I}(z_1) \omega_{\phi_I}(z_2)$ plus a term exponentially decaying
off the diagonal like $C_2(t)e^{- kD_I(z_1,z_2)}$.

\subsection{\label{tinfty}The limit as $t \to \infty$ for fixed $k$} 
Now we apply steepest descent to the second integral as $t\to\infty$, and keeping $k, N_k$ fixed. We obtain
\begin{align}
\nonumber \p_\rho I_{2,k}(\infty, \rho)
: = 
\lim_{t\to\infty}\p_\rho I_{2,k}(t, \rho)=\lim_{t \to \infty}\frac{2t}\rho-\frac1\rho\bigl(2t+\log(1-\rho)+\mathcal O(1/t)\bigr)=-\frac{\log(1-\rho)}{\rho}.
\end{align}
Thus we have,
\begin{equation}
\label{tinfty1}
I_{2,k}(\infty, \rho)={\rm Li}_2(\rho).
\end{equation}
As mentioned above, this is the same formula as \eqref{QN}. 
In the scaling limit around the diagonal with pairs of points of the form $(z, z + \frac{u}{\sqrt{k}})$ we have $\rho\simeq e^{-|u|^2}$.  In the next section we connect this limit with the correlations between zeros of Gaussian random holomorphic sections, see \eqref{smooth1} for a more precise statement.

\subsection{\label{METSCALINGLT} The metric scaling limit with $t \to t \epsilon_k^{-2}$ }
The goal now is to evaluate $I_{2, k}( \epsilon_k^{-2} t, \rho)$ asympotically as $k \to \infty$. This scaling 
keeps the $d_k$-balls of uniform size as $k \to \infty$ with respect to the limit
Mabuchi metric. Thus, as $k$ changes the Brownian motion with respect to
$g_k$ probes distances of size $t$ from the initial metric $\omega_0$ for all $k$.

In order to apply steepest descent for $\lambda \in [0, \infty]$ we change variables $\lambda \to \epsilon_k^{-2} \lambda$ in
\eqref{MAIN1} so that
the exponent becomes
$$ \epsilon_k^{-2} \left(- \frac{\lambda^2}{2t} + \lambda\right). $$
The saddle point occurs at $\lambda = t$,  and the critical value of the exponent
is $\frac{t}{2}$. The prefactor $e^{-t/2}$ in front of the integral, with $t \to \epsilon_k^2 t$
again cancels the critical value of the phase and the singular term. The integrand in \eqref{MAIN1} at the saddle point is asymptotic to
\begin{align}
&\frac{1}{\sqrt{\coth^2 ( \epsilon_k^{-2}t)-\rho}}\log\frac{\sqrt{\coth^2( \epsilon_k^{-2} t)-\rho}+\sqrt{1-\rho}}{\sqrt{\coth^2( \epsilon_k^{-2} t)-\rho}-\sqrt{1-\rho}}
\simeq \frac{2\epsilon_k^{-2} t}{ \sqrt{1-\rho}}  + \frac{1}{ \sqrt{1-\rho}} \log (1 - \rho) + \cdots,
\end{align}
so $$\partial_{\rho} I_{2,k}(\epsilon_k^{-2} t, \rho)  \simeq -\frac{\log (1 - \rho)}{\rho},$$ 
exactly as in \eqref{tinfty1}.
\subsection{Spatial scaling}

In this section, we consider the scaling asymptotics of $I_{2, k}(\rho)$ discussed
in \S \ref{MAIN}. The only
new element in the calculation is the scaling asymptotics of the Berezin kernel
$\rho$. The calculation of $I_{2, k}(\rho)$ above does not change.

The main input into the scaling asymptotics is the following facts about the Berezin
kernel (see \cite{ShZ,ShZ2} for background and references).
Off the diagonal one has
\begin{equation}
\label{kdependence} \rho \simeq e^{-kD_I(z_1,z_2)},
\end{equation}
where the Calabi diastasis function is given by
\begin{equation}
D_I(z_1,z_2)= \phi_I(z_1, z_1) + \phi_I(z_2,z_2) - \phi_I(z_1, z_2) - \phi_I(z_2,z_1).
\end{equation} 
Here $\phi_I(z,\bz)$ is the \kahler potential \eqref{KP} of the background Bergman metric
$\omega_{\phi_I}$, corresponding to the identity matrix $P=I$, and $\phi_I(z_1,z_2)$ is its off-diagonal analytic extension.

The Berezin kernel has a scaling limit on the `near-diagonal'
 where the distance
$d(z_1,z_2)$ between
$z_1$ and
$z_2$ satisfies an upper bound $ d(z_1,z_2)\le b\left(\frac
{\log k}{k}\right)^{1/2}$ ($b\in\R^+$). By comparison, $\rho \to 0$ rapidly on the `far off-diagonal' where  $ d(z_1,z_2)\ge b\left(\frac
{\log k}{k}\right)^{1/2}$, in the sense that for all $b, R > 0$,
$$ \nabla^j
P_k(z_1,z_2)=O(k^{-R})\qquad \mbox{uniformly for }\ d(z_1,z_2)\ge
b\,\sqrt{\frac {\log k}{k}} \;.$$
Here, $\nabla^j$ stands for the $j$-th covariant derivative.

The scaling asymptotics of this kernel near the diagonal may be described as follows.
Let  $ z\in M$.  Then
\begin{equation} \label{BERSCALING} \textstyle  P_k\left(z+\frac u{\sqrt k},z+\frac v{\sqrt k}\right) =
e^{-\frac 12 |u-v|^2}[1 + R_k(u,v)]\;, \end{equation} where
$$ \begin{array}{c}|R_k(u,v)|\le \frac {C_2}2\,|u-v|^2k^{-1/2+\ep}\,, \quad
|\nabla R_k(u)|
\le C_2\,|u-v|\,k^{-1/2+\ep}\,,
\\[8pt] |\nabla^jR_k(u,v)|\le C_j\,k^{-1/2+\ep}\quad j\ge 2\,,\end{array}$$
for $|u|+|v|<b\sqrt{\log k}$.

It follows that the scaling asymptotics of the variance term for random metrics
 is  given by 
\begin{equation} \label{SCALING} I_{2, k}\bigl(\epsilon_k^{-2} t, P_k(z, z + u/\sqrt{k})\bigr) \simeq {\rm Li}_2 (e^{-|u|^2}),  \end{equation}
just as in the limit as $t \to \infty$ first, considered in \S \ref{tinfty}.

\subsection{Energy entropy scaling}

Another natural scaling comes directly from the density \eqref{mutdef}. We separate
out the `action'  from the `amplitude' and express it in terms of the empirical measures
$d\mu_{\lambda}= \frac{1}{N_k} \sum_{j = 1}^{N_k} \delta_{\lambda_j} $ of the eigenvalues
of $\Lambda$
\begin{align}\label{mutdefa} \Delta(\lambda)\nonumber
&\Delta(e^\lambda)e^{-\frac1{4t}\sum_{j=1}^{N}\lambda_j^2}=\\&= e^{ N_k^2 \left(\int_M \int_M
\log |x - y| d\mu_{\lambda} (x) d\mu_{\lambda}(y) + 
 (\int_M \int_M
\log |e^x - e^y| d\mu_{\lambda} (x) d\mu_{\lambda}(y) ) \right)+ \frac{N_k}{2 t} \int_M
x^2 d\mu_{\lambda} }.
 \end{align}
To give all terms the same order in $N_k$ we need to rescale
the time to $t \to t/N_k$. Then the scaled measures $\mu_k$ satisfy
a large deviations principle with the rate function given by the exponent of \eqref{mutdefa}.

\section{\label{BOUNDARY}  The geodesic boundary of $\bcal_k$ and configurations of zeros}
In this section, we explain how the result of
\S \ref{tinfty} is essentially the same as the theory of zeros of random holomorphic
sections. As $t \to \infty$,  the mass of the heat kernel gets concentrated on a part
of the boundary of $\pcal_{N_k}$ corresponding to `singular metrics' given by 
zero sets of holomorphic sections.

Symmetric spaces of non-compact type have several different notions of boundary and several types of
compactifications. For background we refer to \cite{GLT}.
The boundary relevant to the heat kernel measures and their $t \to \infty$ limit
is best stated in terms of the Bergman metrics themselves and their limits
along geodesic rays of $\pcal_{N_k}$.  

\begin{maindefn} The weak* compactification of $\bcal_k$  is $\bcal_k \cup \partial
\bcal_k$ where $\partial \bcal_k$ is the set of limit points  (i.e. endpoints) of the
Bergman metrics along  Bergman geodesic rays $\omega_k(s)$. 

\end{maindefn}

In fact, the only relevant boundary points are the  ones arising from the geodesic
ray starting at the balanced metric $x_k \simeq 0$ and with initial velocity in
the direction of $\delta_N$, together with the endpoints of the $U(N_k)$-orbit of this ray.

A geodesic ray in the space of Bergman potentials is a one-parameter family of metrics whose
potentials have the form,
 $$\beta_t = \frac{1}{k}  \log \sum_j e^{t \lambda_j} |s^U_j(z)|^2, $$
where in $SL(N_k,\C)/SU(N_k)$  the ray starts at the origin and  has initial vector $(U, \Lambda)$. 
We note that
\begin{equation} \label{DIFF}  \frac{1}{k} \log \sum_j e^{t \lambda_j} |s^U_j(z)|^2 = \frac{t \lambda_{\max} }{k} + \frac{1}{k} \log  \left(|s^U_{\max}(z)|^2
+   \sum_{j \not= \max} e^{t (\lambda_j - \lambda_{\max})} |s^U_j(z)|^2\right). \end{equation}
Here, $\lambda_{\max}$ is the largest of the $\lambda_j$ and $s^U_{\max}$ is
the corresponding section.
Clearly,

\begin{equation} \label{DIFFb} \begin{array}{l}  \frac{1}{k} \log \sum_j e^{t \lambda_j} |s^U_j(z)|^2 - \sup_M  \frac{1}{k} \log \sum_j e^{t \lambda_j} |s^U_j(z)|^2 
\\ \\=
  \frac{1}{k} \log  \left(|s^U_{\max}(z)|^2
+ \sum_{\lambda_j \not= \max} e^{t (\lambda_j - \lambda_{\max})} |s^U_j(z)|^2 \right) \\ \\ - \sup_M \frac{1}{k} \log \left(  |s^U_{\max}(z)|^2
+  \sum_{\lambda_j \not= \max} e^{t (\lambda_j - \lambda_{\max})} |s^U_j(z)|^2 \right)  . \end{array} \end{equation}
This family of potentials is bounded above by $0$ and is pre-compact in $L^{p} (M)$ for all $1 \leq p < \infty$. The $\sup$ in the second term is itself
bounded above by $\frac{\log k}{k}$ and therefore we may remove the sup without changing the limit (i.e. the sup was only needed to get rid of the 
$ \frac{t \lambda_{\max} }{k}$ term). We then observe that for any Bergman geodesic ray, 
\begin{equation} \label{L1} \lim_{t \to \infty} \left\|\frac{1}{k} \log  \left(|s^U_{\max}(z)|^2
+ \sum_{\lambda_j \not= \max} e^{t (\lambda_j - \lambda_{\max})} |s^U_j(z)|^2 \right) 
- \frac{1}{k} \log |s_{\max}(z)|^2  \right\lVert_{L^1(M)} \to 0.  \end{equation}
Indeed, 
\begin{equation} \label{Ft} F_t( z): =  \left(|s^U_{\max}(z)|^2
+ \sum_{\lambda_j \not= \max} e^{t (\lambda_j - \lambda_{\max})} |s^U_j(z)|^2 \right)   \end{equation}
is monotonically decreasing to $|s^U_{\max}(z)|^2$ for each $z$. Therefore its logarithm monotonically decreases to $\log |s^U_{\max}(z)|^2$.  If we subtract 
$ \frac{1}{k} \log |s_{\max}(z)|^2 $ then the difference is always $\geq 0$ and we may remove the absolute values, and then apply the monotone convergence
theorem to take the limit under the integration sign. But the limit equals zero almost everywhere since each term tends to zero, proving \eqref{L1}.

We now consider the $(1,0)$ forms obtained by taking $\partial$ of the potentials. 
Taking $\partial$ kills the $\frac{t \lambda_{\max} }{k}$ term, and gives,
\begin{equation} \label{DIFF2}  \frac{1}{k} \partial \log \sum_j e^{t \lambda_j} |s^U_j(z)|^2 = \frac{1}{k} \partial \log \left(  |s^U_{\max}(z)|^2 +
   \sum_{\lambda_j \not= \max} e^{t (\lambda_j - \lambda_{\max})} |s^U_j(z)|^2 \right). \end{equation}

\begin{prop} The weak  limits
of Bergman metrics along geodesic rays are generically given by the normalized zero
distributions of holomorphic sections of  $L^k$, i.e. as $t 
\to \infty$, 
$$\frac{1}{k} \ddbar \log \sum_j e^{t \lambda_j} |s^U_j(z)|^2  \to
\frac{1}{ k} \ddbar\log  |s|^2. $$
If the highest weight has multiplicity $r$, then the  limit is
$\frac{1}{ k} \ddbar\log \sum_{j = 1}^r |s_j|^2 $
where $1 \leq r \leq n$ and $\{s_j\}_{j =1 }^r$ is any set of sections in $H^0(M, L^k)$. \end{prop}

 \begin{proof} We take $\ddbar$ of the two terms of 
 \eqref{L1}  to obtain,

\begin{equation} \label{DIFF3}  \frac{1}{k} \ddbar \log \sum_j e^{t \lambda_j} |s^U_j(z)|^2 
\to  \frac{1}{k} 
\ddbar  \log  |s^U_{\max}(z)|^2
 \end{equation}
 in the sense of distributions. That is,  if  we integrate against a smooth $(n-1, n-1)$ form $\psi$ (i.e. a smooth function if the dimension of $M$ equals one),
then 
$$\begin{array}{l} \int_M \psi \wedge  \frac{1}{k} \ddbar \log 
\left(  |s^U_{\max}(z)|^2
+ \sum_{j \not= \max} e^{t (\lambda_j - \lambda_{\max})} |s^U_j(z)|^2 \right)  \\ \\
= \int_M  \ddbar \psi \wedge  \frac{1}{k} \log  \left(|s^U_{\max}(z)|^2
+ \sum_{\lambda_j \not= \max} e^{t (\lambda_j - \lambda_{\max})} |s^U_j(z)|^2 \right)   \end{array}$$
and by \eqref{L1} the right side tends 
$\int_M 
\ddbar \psi \wedge \frac{1}{k} \log  |s^U_{\max}(z)|^2. $

Multiplicity $r$ means that $ \lambda_{\max}$ has multiplicity $r$, and then one
sums over the associated sections $s_j$. 
\end{proof}

A key point for this article is that each of the weights in $\delta_N$ is distinct. 
Hence $r = 1$ in that direction, and therefore the boundary points in the direction
$\delta_N$ and its $U(N)$-translates all consist of $\delta$-functions
$\frac{1}{k} \ddbar \log |s(z)|^2$ of holomorphic sections of $L^k$. Every section
arises in the $U(N_k)$-orbit.

\subsection{\label{RANDOMZEROS} Relation to zeros of holomorphic sections}

In \cite{ShZ}, the authors found a  bi-potential $Q_N\in
L^1(M\times M)$ for the pair correlation function of zeros of Gaussian random
holomorphic sections
\begin{equation}\label{gaussian0}d\gamma(s)=\frac{1}{\pi^m}e^
{-|c|^2}dc\,,\qquad s=\sum_{j=1}^{n}c_js_j\,,\end{equation} on
$L^k$, where $\{s_j\}$ is an orthonormal basis  and
$dc$ is $2n$-dimensional Lebesgue measure. This Gaussian is
characterized by the property that the $2n$ real variables $\Re
c_j, \Im c_j$ ($j=1,\dots,n$) are independent Gaussian random
variables with mean 0 and variance $\half$; i.e.,
$$\E c_j = 0,\quad \E c_j c_k = 0,\quad  \E c_j \bar c_k =
\de_{jk}\,.$$

The
current of integration $Z_s$ over the zeros of one section 
is given by the Poincar\'e-Lelong
formula
\begin{equation} Z_{s} =
\frac{\sqrt{-1}}{ \pi } \partial \bar{\partial}\log |f| = \frac{\sqrt{-1}}{
\pi } \partial
\bar{\partial}\log\left\|s^N\right\|_{h} + \phi_I(z)\;,
\label{Zs}
\end{equation} 
As mentioned above, 
the limit as $t \to \infty$ of random Bergman metrics along the rays above must
be random singular metrics, and we claim that the limit ensemble is equivalent 
to the Gaussian one.   Indeed,
the limit measure is $U(N_k)$-invariant and there exists just one such measure
up to equivalence,  namely the Gaussian measure above. This explains why the
 bipotential $4\pi^2Q_k$ of \cite{ShZ} is the same as the $ t \to \infty$ limit of the two-point function  $I_{2, k} (\rho)$  of the heat kernel
ensemble in \S \ref{tinfty}, corrobrating that this limit ensemble is that of random
zeros of sections.

\subsection{\label{AN}Smooth linear statistics and asymptotic normality}

Similar to the area random variable $X_U$ but easier to work with, is the smooth linear statistic
\begin{equation} \label{Xphi} X_{f}(\omega) = \int_M f \omega, \end{equation}
where  $f \in C^{\infty} (M)$ is a smooth test function. It has a much smaller variance
than $X_U$ since the effect of `zeros along the boundary' is smoothed out. In
\cite{ShZ2} it is proved that 
\begin{equation}\var\big(X_{f} \big) = k\inv
\left[\frac{\zeta(3)}{16\pi}\,\|\Delta f\|_2^2
 +O(k^{-\frac 12 +\ep})\right]\;.\label{smooth1}\end{equation}

It was proved by  Sodin-Tsirelson
\cite{ST} in certain cases and then \cite{ShZ} for the present setting that
the random variables $X^k_{f}$ obey an asymptotic central limit theorem. Namely,
if we normalize $X_{f} $ to have mean zero and variance one, then
\begin{equation}\label{Xphi2} \frac{X_{f} -\E(X_{f})}{\sqrt{\var(X_{f})}} \end{equation}
tends in distribution to the
standard Gaussian distribution
$\ncal(0, 1)$ as $ k \to \infty$. That is,
$$
k^{1/2} \left(X_{f}-\frac k\pi\ \int_M f \om_{\phi}\right) \to 
\ncal(0, \sqrt{\kappa_1}\, \|\ddbar f\|_2)$$
in the weak sense of convergence of distributions as $ k \to \infty$, where
$\kappa_1$ is a certain positive universal constant and
 $\ncal(0, \sigma)$ denotes the (real)
Gaussian distribution of mean zero and variance
$\sigma^2$.

The same results  hold in the limiting case of the heat kernel measures when
$t \to \infty$. It would be interesting to investigate the analogous variance and asymptotic
normality results for heat kernel measures in the other regimes as $k \to \infty, t \to \infty$. The formulae of this
article combined with the techniques of \cite{ShZ, ShZ2}  make this possible.

\end{document}